\documentclass[11pt, oneside]{article}   	
\usepackage{geometry}                		
\usepackage{graphicx}				
\usepackage{amssymb}
\usepackage{amsmath}
\usepackage{mathtools}
\usepackage{verbatim}
\usepackage{float}

\newtheorem{theorem}{Theorem}[section]
\newtheorem{lemma}[theorem]{Lemma}

\newenvironment{proof}[1][Proof]{\begin{trivlist}
\item[\hskip \labelsep {\bfseries #1}]}{\end{trivlist}}
\newenvironment{definition}[1][Definition]{\begin{trivlist}
\item[\hskip \labelsep {\bfseries #1}]}{\end{trivlist}}

\title{A Constructive Proof of Jacobi's Identity for the Sum of Two Squares}
\author{Mario DeFranco}

\begin{document}
\maketitle

\abstract{We present a constructive proof of Jacobi's identity for the sum of two squares. We present a combinatorial proof of the Jacobi Triple Product and combine with a proof of Hirschhorn to define an algorithm. The input is a factorization $n=dN$ with $d \equiv1\mod 4$ plus two bits of data, and whose output is either another factorization $n=d'N'$ and $d' \equiv3\mod 4$ with two more bits of data, or a pair of integers whose squares sum to $n$. We phrase this algorithm in terms of integer partitions and matchings on an infinite graph.} 

\section{Introduction}

Jacobi's identity for the sum of two squares \cite{Jacobi} states that for positive integer $n$
\begin{equation} \label{Jacobi}
r_2(n) = 4(d_1(n) - d_3(n)),
\end{equation}
where $r_2(n)$ is the number of ordered pairs of integers $(m_1, m_2)$ such that 
\[
m_1^2+m_2^2 = n
\]
and
\begin{align*}
d_1(n) &= \text{the number of divisors of $n$ congruent to 1 mod 4}\\ 
d_3(n) &= \text{the number of divisors of $n$ congruent to 3 mod 4}.
\end{align*}
This identity is equivalent to the Lambert series 
\begin{equation}\label{Lambert}
\theta_3(q)^2 = 1+4\sum_{n=1}^\infty (\frac{q^{4n-3}}{1-q^{4n-3}} - \frac{q^{4n-1}}{1-q^{4n-1}})
\end{equation}
where 
\[
\theta_3(q) = \sum_{n=-\infty}^\infty q^{n^2}
\]
and where we expand the rational functions as power series 
\[
\frac{x}{1-x} = x \sum_{m=0}^\infty x^m.
\]
There exist combinatorial proofs of \eqref{Lambert} in the literature; see Borwein and Borwein \cite{Borwein} section 9.1 and M. D. Hirschhorn \cite{Hirschhorn}. These proofs compare the coefficients of both sides of \eqref{Lambert} and show they are equal; in other words, they show that the orders of two finite sets are the same. The proof we present here proves \eqref{Lambert} by explicitly defining a bijection between those sets. Specifically, we denote for integer $n \geq0$ the sets
\begin{align}
\mathrm{SquarePairs}(n) &= \{ (m_1,m_2) \in \mathbb{Z}^2: m_1^2 + m_2^2 = n\} \nonumber \\ 
\mathrm{Factors1mod4}(n)  &= \{ (d,N; \epsilon_1, \epsilon_2): \epsilon_1,\epsilon_2 \in \{ 0,1\}; d,N \in \mathbb{N}, d \equiv 1 \text{ mod } 4  \text{ and } dN=n\} \nonumber \\
\mathrm{Factors3mod4}(n)  &= \{ (d,N; \epsilon_1, \epsilon_2): \epsilon_1,\epsilon_2 \in \{ 0,1\}; d,N \in \mathbb{N}, d \equiv 3 \text{ mod } 4  \text{ and } dN=n\}. \label{sets}
\end{align}
We define a bijection $P$
\[
P: \mathrm{Factors1mod4}(n) \rightarrow \mathrm{SquarePairs}(n)  \cup \mathrm{Factors3mod4}(n).
\]
To define $P$, we first present a combinatorial proof of the Jacobi Triple Product using integer partitions. The proof of \eqref{Lambert} by Hirschhorn \cite{Hirschhorn} uses the Jacobi Triple Product as a starting point. We then define three graphs whose vertices are indexed by lists of partitions that arise from the proof of the Jacobi Triple Product. These graphs codifies the proof strategy of \cite{Hirschhorn}. We show that the sets \eqref{sets} correspond to subsets of the final graph $G''$, and construct paths such that each path that begins at a vertex in $\mathrm{Factors1mod4}(n)$ ends at a vertex in $\mathrm{SquarePairs}(n) \cup \mathrm{Factors3mod4}(n)$. 

We now establish notation for partitions and graphs. An integer partition, or just partition, is a finite multi-set of positive integers. Each element of this multi-set is called a part. We denote a partition $\lambda$ by the notation
\[
\lambda = \prod_{i=1}^\infty (q^i)^{\lambda(i)}
\]
where $\lambda(i)$ is the number of parts of $\lambda$ that are equal to $i$ and only finitely many of $\lambda(i)$ are non-zero. If each non-zero $\lambda(i)=1$, then the multi-set is actually a set and we say that the partition has distinct parts. We let 
\[
\lambda(q) = q^{\sum_{i=1}^\infty i\lambda(i)}
\]
denote that monomial $\mathbb{Z}[q]$. 
We denote
\[
\Sigma(\lambda) = \sum_{i=1}^\infty i\lambda(i)
\]
and 
\[
|\lambda| = \sum_{i=1}^\infty \lambda(i).
\]
For a partition $\alpha$ with distinct parts, we will also denote the partition as an ordered finite sequence 
\[
\{ a_1, a_2, ..., a_l\}
\]
such that 
\[
0<a_i < a_{i+1}.
\]

A graph $G$ is a pair $(V,E)$, where $V$ is a (possibly infinite) set and $E$ is a set of distinct unordered pairs $\{v_1,v_2\}$ for $v_1,v_2 \in V$. We call $V$ the vertex set or just vertices of $G$ and $E$ the edges of $G$. For a graph $G$, we say that a matching $\mathcal{M}$ of $G$ is a subset of the edges of $G$ such that each vertex of $G$ is in at most one edge of $\mathcal{M}$. A perfect matching is a matching such that each vertex is in exactly one edge. 

\section{Proof of Jacobi Triple Product}\label{Triple Product}
Jacobi Triple Product: 

\[
\prod_{m=1}^\infty (1-q^{2m})(1+z q^{2m-1})(1+z^{-1} q^{2m-1}) = \sum_{n=-\infty}^\infty z^n q^{n^2}
\]

\begin{proof}
Divide by $$ \prod_{m=1}^\infty (1-q^{2m})$$ and expand each factor as a power series to obtain 
\begin{equation} \label{expanded}
\prod_{m=1}^\infty (1+z q^{2m-1})(1+z^{-1} q^{2m-1}) = \sum_{n=-\infty}^\infty z^n q^{n^2}(\prod_{j=1}^\infty \sum_{k=0}^\infty (q^{2j})^k). 
\end{equation}

Expanding the terms on the left, the coefficient of $z^n$ will be a sum of terms which are clearly indexed by an ordered pair $(A,B)$, where $A$ and $B$ are each a finite (possibly empty) set of 
odd positive integers, though $A \cap B$ may be non-empty, and $|A|-|B| =n$. To prove \eqref{expanded} we thus give a bijection
\[
(A,B) \mapsto (\lambda(A,B),n)
\]
where $\lambda(A,B)$ denotes a finite (possibly empty) set of positive even integers; that is $\lambda(A,B)$ is a partition with distinct even parts.  
We consider the case $|A|\geq |B|$. For if $|A|<|B|$, we set 
\[
(A,B) \mapsto (\lambda(B,A),|A|-|B|).
\] 
Let $A = (a_1, a_2, ..., a_l)$ and $B = (b_1, b_2, ..., b_{l-n})$, with $a_i<a_{i+1}$ and $b_i<b_{i+1}$.  Thus $\lambda(A,B)$ must be a partition such that 
\[
n^2+\Sigma(\lambda(A,B)) = \Sigma(A)+\Sigma(B).
\]
where $n = |A|-|B|$.

We define
\begin{equation}\label{bijection}
\lambda(A,B)= \left(\prod_{j=1}^{l} (q^{2(l-j+1)})^{\frac{a_{j}-a_{j-1}-2}{2}}\right)\left(\prod_{j=1}^{l-n} (q^{2l+b_j-(2j-1)})^1\right),
 \end{equation}
 where each factor $(q^{j})^i$ contributes $i$ parts of the number $j$ in the partition $\lambda(A,B)$ of $N$, and $a_0=-1$. 
 
 To see that the map results in a partition of $\Sigma(A)+\Sigma(B) - n^2$, we consider the pair $(A,B)$ with minimal entries, which is 
 \[
 A = (1,3,5,..., 2l-1), \, \, \, \, \, \, B= (1, 3, ..., 2(l-n)-1)
 \]
 Then $\Sigma(A)+\Sigma(B) = n^2+2l(l-n)$ and $(A,B)$ maps to $(q^{2l})^{l-n}$, the partition with $(l-n)$ parts of the integer $2l$. Incrementing each $a_i$ by 2 for each $i\geq j$ corresponds to adding one part of $q^{2(l-j+1)}$ to the partition. Incrementing each $b_i$ by 2 for each $i \geq j$ corresponds to changing each part $q^{2l+b_i-2i+1}$ to $q^{2l+b_i-2i+3}$.

 To define the reverse map, we are given an integer $n\geq0$ 
 and a partition $\lambda$ consisting of all even parts and must construct $(A,B)$. We let 
 \[
  |\lambda|_{>j}
 \]
denote the number of parts of the partition $\lambda$ that are greater than $j$. Then the order $|A|=l$ is determined by $l=n+i$, where $i$ is the unique integer that satisfies both of the inequalities
 \[
 |\lambda|_{>2(n+i-1)} \geq i,  \,\,\,\,\,\,\, |\lambda|_{>2(n+i)} \leq i.
 \]
This ensures that there are at least $(l-n)$ parts of $\lambda$ that are at least equal to $2l$ and at most $(l-n)$ parts that are greater than $2l$. First, the set $B$ is determined by the $(l-n)$ greatest parts of $\lambda$, 
 and then the set $A$ is determined by the parts less than or equal to $2l$. $\square$
\end{proof} 

 The two sets $A$ and $B$ in the map (\ref{bijection}) correspond to two ways of building up partitions; incrementing entries in $A$ adds a certain part (which is $\leq 2l$)  to the partition, and incrementing entries in $B$ increases certain parts  (which are $\geq 2l$) that are already in the partition.   

\section{Graphs corresponding to Hirschhorn's Equation}

\subsection{The vertices of $G$} \label{vertices G}
From \cite{Hirschhorn}, 
we have 
Hirschhorn's equation:
\begin{align} 
& \prod_{n\geq 1} (1+(-aq^{\frac{1}{2}})(q^{\frac{1}{2}})^{2n-1})(1+(-aq^{\frac{1}{2}})^{-1}(q^{\frac{1}{2}})^{2n-1})(1-(q^{\frac{1}{2}})^{2n})\label{V_0 term}\\ 
&-  \prod_{n\geq 1} (1+(a^2q)(q^2)^{2n-1})(1+(a^2q)^{-1}(q^2)^{2n-1})(1-(q^2)^{2n}) \label{V_1 term}\\ 
&-  a^{-1} \prod_{n\geq 1} (1+(a^2q^{-1})(q^2)^{2n-1})(1+(a^2q^{-1})^{-1}(q^2)^{2n-1})(1-(q^2)^{2n}) \label{V_2 term}\\
=&0.
\end{align}
We express this equation as a perfect matching on an infinite graph $G$. The vertices in $V_0$ correspond to terms in the expansion of the product at \eqref{V_0 term}; the vertices in $V_1$ correspond to terms in the expansion of the product at \eqref{V_1 term}; and the vertices in $V_2$ correspond to terms in the expansion of the product at \eqref{V_2 term}.
\begin{definition}
The vertex set of $G$ is the disjoint union of three sets $V_0, V_1, V_2$. The vertices in each $V_i$ are indexed in the same manner by the ordered triple $(A,B,D)$ and $0\leq i\leq 2$:
\[
v = (A,B,D;i) \in V_i
\] 
where $A$ and $B$ are finite (possible empty) sets of distinct odd positive integers as above and $D$ is a finite (possible empty) set of distinct even positive integers; and $i$ indicates that $v \in V_i$.
We use the notation for the elements of $A$ and $B$ as above and let 
\[
D = \{d_1, d_2, ..., d_{|D|}\}.
\] 

Define $wt(v,a,q)$ by 
for $v\in V_0$
\[
wt(v,a,q) = (-1)^{|D|}(-a q^{\frac{1}{2}})^{|A|-|B|}A(q^{\frac{1}{2}})B(q^{\frac{1}{2}})D(q^{\frac{1}{2}}) ; 
\]
for $v \in V_1$ by
\[
wt(v,a,q) = - (-1)^{|D|}(a^{2}q)^{|A|-|B|}(q^2)^{|A|-|B|}A(q^2)B(q^2)D(q^2); 
\]
and for $v \in V_2$ by 
\[
wt(v,a,q) = \frac{ (-1)^{|D|}}{a} (a^{2}q^{-1})^{|A|-|B|}A(q^2)B(q^2)D(q^2); 
\]
$\square$
\end{definition}

\begin{definition}
Define a minimal set $\mathrm{MinSet}(n)$ for integer $n \geq 1$ by
\[
\mathrm{MinSet}(n)=\{ 1,3,5,...,2n-1\}.
\]
Define a vertex $v =(A,B,D;i) \in V_i$ to be terminal if either of the following hold for some integer $n \geq 1$: 
\begin{align*}
&1. \,\,A \text{ is minimal},\,\, B= \emptyset, \, \, D = \emptyset   \\ 
&2. \,\,A=\emptyset,\,\, B \text{ is minimal}, \, \, D = \emptyset   \\
&3. \,\,A=\emptyset, \,\,B= \emptyset,\,\, D=\emptyset. 
\end{align*}
$\square$
\end{definition}

\begin{definition} 
For a set $A$
\[
A=\{a_1,a_2,...,a_l\},
\]
define
\[
\mathrm{increment}(A,i) = \{a_1, a_2,..., a_{i-1}, a_i+2, a_{i+1}+2, ..., a_{l}+2\} 
\]
and 
\[
\mathrm{decrement}(A,i) = \{a_1, a_2,..., a_{i-1}, a_i-2, a_{i+1}-2, ..., a_{l}-2\}. 
\]
$\square$
\end{definition}

From proof above we have $\lambda(A,B) = \emptyset$ if and only if $(A,B)$ is of the form
\[
(A,B) = (\mathrm{MinSet}(n),\emptyset) \text{ or }  (\emptyset, \mathrm{MinSet}(n))
\] 
for some $n$. 

\subsection{The matching $\mathcal{J}$}
We define the perfect matching $\mathcal{J}$ on $G$.
\begin{definition} \label{J matching}
The perfect matching $\mathcal{J}$ is defined to be an involution on $V$. Let $v = (A,B,D;i)$. Let $\lambda=\lambda(A,B)$.  Assume $|A|\geq |B|$. 
Let $|A|=l$. For cases 1-4 below assume $\lambda$ and $D$ are not both empty. If $D$ is not empty, let $d_1$ denote the smallest integer in $D$. If $\lambda$ is not empty, let $\lambda_1$ denote the smallest integer in $\lambda$. 

\vspace{3mm}

\noindent1. If $\lambda$ is empty or if $d_1\leq\lambda_1$, and if $d_1 > 2l$, then that means $A = \mathrm{MinSet}(k)$ for some $k$. Define
\[
\mathcal{J}(v) = (\mathrm{MinSet}(k+1), B \cup \{ d_1-1-2l\}, D\backslash \{d_1\} ;i).
\] 
2. If $\lambda$ is empty or if $d_1 \leq\lambda_1$, and if $d \leq 2l$, then define 
\[
\mathcal{J}(v) = (\mathrm{increment}(A,\frac{2l-d_1+2}{2}), B, D\backslash \{d_1\} ;i).
\]
3. If $D$ is empty or if $d_1>\lambda_1$, and if $A = \mathrm{MinSet}(k)$ for some $k$, then define 
\[
\mathcal{J}(v) = (\mathrm{MinSet}(k-1), B \backslash \{ b_1\}, D\cup \{\lambda_1\} ;i).
\] 
4. If $D$ is empty or if $d_1>\lambda_1$, and if $A$ is not minimal, then define 
\[
\mathcal{J}(v) = (\mathrm{decrement}(A,\frac{2l-\lambda_1+2}{2}), B, D\cup \{\lambda_1\} ;i).
\]
5. If $\lambda$ and $D$ are both empty, then for integer $m \geq 0$ define 
\[
\mathcal{J}((\mathrm{MinSet}(2m), \emptyset,\emptyset;0)) =  (\mathrm{MinSet}(m), \emptyset,\emptyset;1),
\] 
\[
\mathcal{J}((\emptyset, \mathrm{MinSet}(2m),\emptyset;0)) = ( \emptyset,\mathrm{MinSet}(m),\emptyset;1), 
\]
\[
\mathcal{J}((\mathrm{MinSet}(2m-1), \emptyset,\emptyset;0)) = ( \mathrm{MinSet}(m),\emptyset,\emptyset;2), 
\]
and
\[
\mathcal{J}((\mathrm{MinSet}(2m), \emptyset,\emptyset;0)) =( \emptyset,\mathrm{MinSet}(m),\emptyset;2). 
\]
If $|B|>|A|$, then 
with 
\[
\mathcal{J}((B,A,D;i)) = (B',A',D';i')
\]
define 
\[
\mathcal{J}((A,B,D;i)) = (A',B',D';i').
\]
In cases 1 and 3, we say that $v$ is order-preserved since $|A|$ and $|B|$ are the same in $v$ and $\mathcal{J}(v)$. In cases 2 and 4, we say that $v$ is not order-preserved. 
$\square$
\end{definition}
\begin{lemma}\label{J is perfect}
The set of edges $\{v,\mathcal{J}(v)\}$ for $v\in V$ is a perfect matching and 
 \[
 wt(v,a,q) + wt(\mathcal{J}(v),a,q) = 0.
 \] 
\end{lemma}
\begin{proof}
It is straightforward to check using the proof of the Jacobi Triple Product that $\mathcal{J}$ is an involution on $V$ and thus a perfect matching. It is also straightforward to check in each case that  
\[
wt(v,a,q) + wt(\mathcal{J}(v),a,q) = 0.
\]

\end{proof}
\subsection{The vertices of $G'$}
Now we differentiate Hirschhorn's equation with respect to $a$. This corresponds to creating a new graph $G'$ whose vertex set $V'$ is defined next. Differentiating by the product rule means each factor $a$ from the set $A$ and each $a^{-1}$ from $B$ contributes a separate vertex. That is why, for example, each element of  the set $A$ for vertices in $V_0$ contributes a vertex indexed by $j$; but for vertices in $V_1$ or $V_2$, each element $A$ comes with a power of $a^2$, which corresponds to two vertices indexed by $j$ and $j+1$. The $j=0$ index arises from the factor of $a^{-1}$ in the $wt$ function for $V_2$; this $a^{-1}$ does not come from $A$ or $B$. 

\begin{definition}
We define the vertex set $V'$ of $G'$ as a disjoint union of three sets $V_0', V_1',$ and $V_2'$. We say that a vertex $v \in V_i$ contributes a finite number of vertices of the form $v(j,+)$ or $v(j, -) \in V_i'$, where $j$ is an integer we specify below. 

Let 
\[
v= (A,B,D; 0) \in V_0.
\]
We denote 
\[
v(j,+) = (A,B,D; 0)(j,+) \in V_0'
\]
for $j \in A$, and 
\[
v(j,-) = (A,B,D; 0)(j,-) \in V_0'
\]
for $j\in B$. The vertex 
\[
(\emptyset, \emptyset, \emptyset;0) 
\]
does not contribute any vertices to $V_0'$.

Let 
\[
v= (A,B,D; 1) \in V_1.
\]
We denote
\[
v(j,+) = (A,B,D; 1)(j,+) \in V_1'
\]
for $j \in A$ or $j-1 \in A$, and   
\[
v(j,-) = (A,B,D; 1)(j,-) \in V_1'
\]
for $j \in B$ or $j-1 \in B$. The vertex 
\[
(\emptyset, \emptyset, \emptyset;1) 
\]
does not contribute any vertices to $V_1'$.

Let 
\[
v= (A,B,D; 2) \in V_2.
\]
We denote
\[
v(j,+) = (A,B,D; 2)(j,+) \in V_2'
\]
for $j \in A$ or $j-1 \in A$, and 
\[
v(j,-) = (A,B,D; 2)(j,-) \in V_2'
\]
for $j \in B$ or $j-1 \in B$, and also 
 \[
v(0) = (A,B,D; 2)(0) \in V_2'
\]
for any choice of $A$ and $B$.
Thus the vertex 
\[
(\emptyset, \emptyset, \emptyset;2) \in V_2
\]
contributes to $V_2'$ the vertex 
\[
(\emptyset, \emptyset, \emptyset;2)(0).
\]
$\square$ 
\end{definition}
We extend the weight function $wt$ to $V'$:
\begin{definition} 
For $v\in V$, define 
\[
wt'(v(j,+),q) = wt(v,1,q),
\] 
\[
wt'(v(j,-),q) = -wt(v,1,q).
\]
and 
\[
wt'(v(0),q)=-wt(v,1,q).
\]
\end{definition} 
\subsection{Extending $\mathcal{J}$ to $G'$}
We extend the perfect matching $\mathcal{J}$ to $G'$.
\begin{definition}

\hspace{1mm} \vspace{2 mm}

 $\mathbf{v\in V_0}$, \textbf{not terminal}. Suppose $v\in V_0$ is not terminal and is order-preserved. Without loss of generality suppose $A$ or $B$ gets incremented. Then if $B$ is the set that gets incremented, or if $A$ is the set that gets incremented and the incremented elements are all greater than $j$, then set 
\[
\mathcal{J}(v(j,+)) = \mathcal{J}(v)(j,+).
\] 
Otherwise set 
\[
\mathcal{J}(v(j,+)) = \mathcal{J}(v)(j+2,+).
\]
For  $v(j,-)$, apply the above definition with $A$ and $B$ switched.

Now suppose $v\in V_0$ is not terminal and is not order-preserved. Without loss of generality suppose $A=\mathrm{MinSet}(k)$ in $v$, and becomes $\mathrm{MinSet}(k-1)$ in $\mathcal{J}(v)$. Let $b_1$ be the smallest element of $B$. We set 
\[
\mathcal{J}(v(2k-1,+)) = v(b_1,-).
\]
Switching $A$ and $B$ in the above definition and letting $a_1$ denote the smallest element of $A$, we set 
\[
\mathcal{J}(v(2k-1,-)) = v(a_1,+).
\]
For all other cases we set 
\[
\mathcal{J}(v(j,+)) =  \mathcal{J}(v)(j,+) \text{ and } \mathcal{J}(v(j,-)) =  \mathcal{J}(v)(j,-).
\]

$\mathbf{v\in V_1 \cup V_2}$, \textbf{not terminal}. Suppose $v$ is not terminal and $v$ is order-preserved. Without loss of generality suppose $A$ or $B$ gets incremented. Then if $B$ is the set that gets incremented, or if $A$ is the set that gets incremented and the incremented elements are all greater than $j$, then set 
\[
\mathcal{J}(v(j,+)) = \mathcal{J}(v)(j,+).
\] 
Otherwise set 
\[
\mathcal{J}(v(j,+)) = \mathcal{J}(v)(j+2,+).
\]
For  $v(j,-)$, apply the above definition with $A$ and $B$ switched.

Suppose $v$ is not terminal and not order-preserved. Without loss of generality suppose $A=\mathrm{MinSet}(k)$ in $v$ and becomes $\mathrm{MinSet}(k-1)$ in $\mathcal{J}(v)$. Let $b_1$ be the smallest element of $B$. We set 
\[
\mathcal{J}(v(2k-1,+)) = v(b_1,-) \text{ and } \mathcal{J}(v(2k,+)) = v(b_1+1,-).
\]
Switching $A$ and $B$ in the above definition and letting $a_1$ denote the smallest element of $A$, we set 
\[
\mathcal{J}(v(2k-1,-)) = v(a_1,+) \text{ and } \mathcal{J}(v(2k,-)) = v(a_1+1,+).
\]
For all other cases we set 
\[
\mathcal{J}(v(j,+)) =  \mathcal{J}(v)(j,+) \text{ and } \mathcal{J}(v(j,-)) =  \mathcal{J}(v)(j,-).
\]
For $v \in V_2$ and $j=0$ with $v$ not terminal, we set 
\[
\mathcal{J}(v(0)) = \mathcal{J}(v)(0). 
\]

$\mathbf{v \in V_0}$ \textbf{is terminal}.

\vspace{1mm}
1.If $v \in V_0$ is terminal with $v = (\mathrm{MinSet}(2m), \emptyset, \emptyset)$, then recall $\mathcal{J}(v) \in V_1$ is terminal with $\mathcal{J}(v) = (\mathrm{MinSet}(m),\emptyset,\emptyset)$. Define
\[
\mathcal{J}(v(j,+)) = \mathcal{J}(v)(\frac{j+1}{2},+)\in V_1'.
\]
for $j$ odd, $1\leq j \leq 2(2m)-1$. 

2.If $v \in V_0$ is terminal with $v = (\emptyset, \mathrm{MinSet}(2m),\emptyset)$, then recall $\mathcal{J}(v) \in V_1$ is terminal with $\mathcal{J}(v) = (\emptyset, \mathrm{MinSet}(m),\emptyset)$. Define
\[
\mathcal{J}(v(j,-)) = \mathcal{J}(v)(\frac{j+1}{2},-)\in V_1'.
\]
for $j$ odd, $1\leq j \leq 2(2m)-1$. 

3.If $v \in V_0$ is terminal with $v = (\mathrm{MinSet}(2m-1), \emptyset, \emptyset)$, then recall $\mathcal{J}(v) \in V_2$ is terminal with $\mathcal{J}(v) = (\mathrm{MinSet}(m),\emptyset,\emptyset)$. Define
\[
\mathcal{J}(v(j,+)) = \mathcal{J}(v)(\frac{j+3}{2},+)\in V_2'.
\]
for $j$ odd, $1\leq j \leq 2(2m-1)m-1$, and define 
\begin{equation} \label{V2 j0 2m-1 pos}
\mathcal{J}((\mathrm{MinSet}(m), \emptyset,\emptyset;2)(0,+) )= \mathcal{J}(\mathrm{MinSet}(m), \emptyset,\emptyset;2)(1,+)
\end{equation}
where both sides of the above equation are elements of $V_2'$.

4. If $v \in V_0$ is terminal with $v = (\emptyset, \mathrm{MinSet}(2m-1), \emptyset)$, then recall $\mathcal{J}(v) \in V_2$ is terminal with $\mathcal{J}(v) = (\emptyset,\mathrm{MinSet}(m),\emptyset)$. Define
\[
\mathcal{J}(v(j,-)) = \mathcal{J}(v)(\frac{j-1}{2},-)\in V_2'.
\]
for $j$ odd, $1\leq j \leq 2(2m-1)m-1$.

Equation \eqref{V2 j0 2m-1 pos} describes the cancellation 
from the differentiation 
\[
a\frac{d}{da}(\frac{1}{a}) = -\frac{1}{a} \frac{d}{da}(a). 
\]
\end{definition}

\begin{lemma}\label{J' is perfect}
The set of edges $\{v',\mathcal{J}(v')\}$ for $v'\in V'$ is a perfect matching and 
 \[
 wt'(v',q) + wt(\mathcal{J}(v'),q) = 0.
 \] 
\end{lemma}
\begin{proof}
This follows from the same properties for $\mathcal{J}$ on $V$. 
\end{proof}

We now define subsets of $V_i'$ that correspond to each term in the differentiation of the infinite products.
\begin{definition} 
For each odd integer $j>0$, define $V_0'(j,+) \subset V_0'$ to be set of a vertices in $V_0'$ of the form $v(j,+)$ and $V_0'(j,-) \subset V_0'$ to be set of a vertices in $V_0'$ of the form $v(j,-)$. 

For each odd integer $j>0$, define $V_1'(j,+) \subset V_1'$ to be set of a vertices in $V_1'$ of the form $v(j,+)$ or $v(j+1,+)$ and define $V_1'(j,-) \subset V_1'$ to be set of a vertices in $V_1'$ of the form $v(j,+)$ or $v(j+1,-)$. 

For each odd integer $j>0$, define $V_2'(j,+) \subset V_1'$ to be set of a vertices in $V_2'$ of the form $v(j,+)$ or $v(j+1,+)$ and define $V_2'(j,-) \subset V_1'$ to be set of a vertices in $V_2'$ of the form $v(j,-)$ or $v(j+1,-)$. Define $V_2'(0) \subset V_2'$ to be set of a vertices in $V_2'$ of the form $v(0)$. $\square$
\end{definition} 

\subsection{The matching $\mathcal{H}$}
\begin{definition}
We define the perfect matching $\mathcal{H}$ restricted to the set $V_0'\backslash V_0'(-1)$.
For a $j \neq -1$, let
\[ 
 v'=(A,B,D;0)(j) \in V_0'\backslash V_0'(-1)
\]
such that $1 \notin B$. Then define
\[
\mathcal{H}(v') = (A,B \cup \{1\},D;0)(j).
\]
$\square$
\end{definition}

\begin{lemma}\label{H is perfect}
The set of edges $\{v,\mathcal{H}(v)\}$ for $v\in V_0' \backslash V_0'(-1)$ is a perfect matching on that vertex set and 
 \[
 wt'(v,q) + wt'(\mathcal{H}(v),q) = 0.
 \] 
\end{lemma}
\begin{proof}
By definition $\mathcal{H}$ is an involution on $ V_0' \backslash V_0'(-1)$. Since the sets of $v$ and $\mathcal{H}(v)$ differ by exactly one element they have different signs but are the same power of $q$ because by construction the element $1$ of $B$ in $V_0$ contributes $q$ to the $0$ power. $\square$
\end{proof}

For vertices in $V_0' \backslash V_0'(-1)$, we apply the matching $\mathcal{H}$ 
to get that 
\[
\sum_{v'\in V_0'} wt'(v',q)=\sum_{v'\in V_0'(-1)} wt'(v',q)=  \prod_{n=1}^\infty (1-q^{\frac{1}{2}} (q^{\frac{1}{2}})^{2n-1}) \prod_{n=2}^\infty (1-q^{-\frac{1}{2}} (q^{\frac{1}{2}})^{2n-1}) \prod_{n=1}^\infty (1- (q^{\frac{1}{2}})^{2n}).
\]

\subsection{The vertices of $G''$}
Now, in the differentiated Hirschhorn's equation, we set $a=1$ and multiply both sides by
\begin{align*}
&\left (\prod_{n=1}^\infty  (1+q^n)^{-1} (1+q^n)^{-1} (1+q^n)(1+q^n)\right)\\ 
&\times \left( \prod_{n=1}^\infty  (1+q^{4n-1})^{-1} (1+q^{4n-3})^{-1} (1+q^{2n})^{-1} (1+q^n)^{-1} (1-q^n)^{-1} \right)\\ 
&\times \left( \prod_{n=1}^\infty  ((1-q^n)(1-q^n)(1-q^n)(1-q^n) (1-q^n)^{-1}) (1-q^n)^{-1} (1-q^n)^{-1} (1-q^n)^{-1}\right)\\ 
=&\left (\prod_{n=1}^\infty  (\sum_{m=0}^\infty (-1)^m (q^n)^m)(\sum_{m=0}^\infty (-1)^m (q^n)^m ) (1+q^n)(1+q^n)\right)\\ 
&\times \left( \prod_{n=1}^\infty (\sum_{m=0}^\infty (-q^{4n-1})^m)(\sum_{m=0}^\infty (-q^{4n-3})^m)(\sum_{m=0}^\infty (-q^{2n})^m)(\sum_{m=0}^\infty (-q^{n})^m)(\sum_{m=0}^\infty (q^{n})^m)  \right)\\ 
&\times \left( \prod_{n=1}^\infty  ((1-q^n)(1-q^n)(1-q^n)(1-q^n)(\sum_{m=0}^\infty (q^n)^m))(\sum_{m=0}^\infty (q^n)^m))(\sum_{m=0}^\infty (q^n)^m))(\sum_{m=0}^\infty (q^n)^m))\right) 
\end{align*}
This corresponds to creating a new graph $G''$. 
In the above expression, consider the  first factor of $\displaystyle (\prod_{n=1}^\infty  (1+q^n)^{-1})$. Expanding each $(1+q^n)^{-1}$ as a geometric series yields a sum of terms. Each term corresponds to a choice of a partition. Call this choice $\mu_1$. Considering the second factor of $\displaystyle (\prod_{n=1}^\infty  (1+q^n)^{-1})$  yields a another choice, call it $\mu_2$. Similarly the first factor of  $\displaystyle (\prod_{n=1}^\infty  (1+q^n))$ yields a partition with distinct parts, call it $\nu_1$. In this way we index the terms in the expansion of the above expression by partitions.  
\begin{definition} 
We say that each vertex $v' \in V'$ contributes infinitely many vertices $v'' \in V''$ of the form 
\[
v'' = v' ( \mu_1, \mu_2, \nu_1, \nu_2, \xi_1,\xi_2, \xi_3, \xi_4,\xi_5,\pi_1, \pi_2, \pi_3, \pi_4, \rho_1,\rho_2, \rho_3, \rho_4)
\]
where  $\mu_1$ and $\mu_2$ are partitions; $\nu_1$ and $\nu_2$ are partitions with distinct parts; $\xi_1$ is a partition with parts congruent to $3 \mod 4$, $\xi_2$ is a partition with parts congruent to $1 \mod 4$,  $\xi_3$ is a partition with even parts, and $\xi_4$ and $\xi_5$ are partitions;  $\pi_k$ are partitions with distinct parts; and $\rho_k$ are partitions. 
Define
\[
V_i'' = \{v'': v' \in V_i' \}.
\]
We denote 
\[
 V_0''(-1) \subset V_0'' 
\]
to be the set of vertices in $V_0''$ that are contributed by vertices in $V_0'(-1) \subset V_0'$. 

Define 
\[
wt''(v'',q) = (-1)^{|\mu_1|+|\mu_2|+\sum_{k=1}^4 |\xi_k|+|\pi_k|} wt(v',1,q) \prod_{k=1}^2 \mu_k(q) \prod_{k=1}^2 \nu_k(q) \prod_{k=1}^5 \xi_k(q) \prod_{k=1}^4 \pi_k(q) \prod_{k=1}^4 \rho_k(q)  
\]
$\square$
\end{definition}

\subsection{Extending $\mathcal{J}$ to G''}
\begin{definition}
For $v'\ \in V'$, let $v'' = v'(\gamma)$ be a vertex that $v'$ contributes to $V''$ where $\gamma$ denotes a particular choice of 17 partitions in the definition of $V''$. Then define
\[
\mathcal{J}(v'(\gamma) =\mathcal{J}(v')(\gamma).
\]
\end{definition}
It follows that $\mathcal{J}$ is a perfect matching on $V''$ and that
\[
wt''(v'')+\mathcal{J}(v'')=0
\]
from those same properties for $V'$. 

\subsection{Operation of pairs of partitions}
\begin{definition}
We define operations $\mathrm{Reciprocal}$, $\mathrm{SqDiffDen}$, $\mathrm{SqDiffNum}$ and $\mathrm{EulerIdentity}$ on pairs of partitions $(\lambda,\mu)$. 

The operation $\mathrm{Reciprocal}(\lambda,\mu)$ corresponds to the cancellation in the equation 
\[
\left(\prod_{n=1}^\infty (1-x^n) \right) \left( \prod_{n=1}^\infty \sum_{m=0}^\infty (x^n)^m  \right)=1.
\]
 
$\mathbf{ \mathrm{Reciprocal}(\lambda,\mu)}$. For a partition $\lambda$ with distinct parts and for any partition $\mu$, suppose there exists a smallest index $n$ such that 
\[
\lambda(n) \neq 0 \text{ or } \mu(n) \neq 0.
\]
If $\lambda(n) \neq 0$, then let 
\[
\lambda' = \left(\prod_{i=1}^{n-1}(q^i)^{\lambda(i)} \right) (q^n)^{0}  \left(\prod_{i=n+1}^{\infty}(q^i)^{\lambda(i)} \right)
\]
\[
\mu' = \left(\prod_{i=1}^{n-1}(q^i)^{\mu(i)} \right) (q^n)^{\mu(n)+1}  \left(\prod_{i=n+1}^{\infty}(q^i)^{\mu(i)} \right).
\]
If $\lambda(n) = 0$, then let 
\[
\lambda' = \left(\prod_{i=1}^{n-1}(q^i)^{\lambda(i)} \right) (q^n)^{1}  \left(\prod_{i=n+1}^{\infty}(q^i)^{\lambda(i)} \right)
\]
and 
\[
\mu' = \left(\prod_{i=1}^{n-1}(q^i)^{\mu(i)} \right) (q^n)^{\mu(n)-1}  \left(\prod_{i=n+1}^{\infty}(q^i)^{\mu(i)} \right).
\]
The define 
\[
\mathrm{Reciprocal}(\lambda,\mu) = (\lambda',\mu').
\]
If no such $n$ exists, then we say that $\mathrm{Reciprocal}(\emptyset, \emptyset)$ is not defined. 

\vspace{1cm}
The operation $\mathrm{SqDiffDen}(\lambda,\mu)$,  corresponds to the cancellation in the equation 
\[
 \left( \prod_{n=1}^\infty \sum_{m=0}^\infty (x^n)^m  \right)\left( \prod_{n=1}^\infty \sum_{m=0}^\infty (-1)^m (x^n)^m  \right)= \left( \prod_{n=1}^\infty \sum_{m=0}^\infty (x^n)^{2m}  \right)
\]
which comes from have a difference of squares in the denominator.

$\mathbf{\mathrm{SqDiffDen}(\lambda,\mu)}$. For two partitions $\lambda$ and $\mu$, suppose there exists a smallest index $n$ such that either 

1. $\lambda(n)\neq 0$ and $\mu(n)$ is even. 

2. $\mu(n)$ is odd.  

\noindent If we have case 1, then let
\[
\lambda' = \left(\prod_{i=1}^{n-1}(q^i)^{\lambda(i)} \right) (q^n)^{\lambda(n)-1}  \left(\prod_{i=n+1}^{\infty}(q^i)^{\lambda(i)} \right)
\]
and
\[
\mu' = \left(\prod_{i=1}^{n-1}(q^i)^{\mu(i)} \right) (q^n)^{\mu(n)+1}  \left(\prod_{i=n+1}^{\infty}(q^i)^{\mu(i)} \right).
\]
\noindent If we have case 2, then let
\[
\lambda' = \left(\prod_{i=1}^{n-1}(q^i)^{\lambda(i)} \right) (q^n)^{\lambda(n)+1}  \left(\prod_{i=n+1}^{\infty}(q^i)^{\lambda(i)} \right)
\]
and
\[
\mu' = \left(\prod_{i=1}^{n-1}(q^i)^{\mu(i)} \right) (q^n)^{\mu(n)-1}  \left(\prod_{i=n+1}^{\infty}(q^i)^{\mu(i)} \right).
\]

Then define 
\[
\mathrm{SqDiffDen}(\lambda, \mu) = (\lambda',\mu').
\] 
\noindent If no such $n$ exists, then we say that $\mathrm{SqDiffDen}(\lambda, \mu)$ is  not defined.

\vspace{1cm}
The operation $\mathrm{SqDiffNum}(\lambda,\mu)$,  corresponds to the cancellation in the equation 
\[
 \left( \prod_{n=1}^\infty (1-q^n)\right)\left( \prod_{n=1}^\infty(1+q^n)\right)= \left( \prod_{n=1}^\infty (1-(q^n)^2)\right)
\]
which comes from have a difference of squares in the numerator.
$\mathbf{\mathrm{SqDiffNum}(\lambda,\mu)}$. Let $\lambda$ and $\mu$ be two partitions each having distinct parts. Suppose there exists a smallest index $n$ such that 
\[
\lambda(n) \neq \mu(n).
\]
Let 
\[
\lambda' = \left(\prod_{i=1}^{n-1}(q^i)^{\lambda(i)} \right) (q^n)^{\mu(n)}  \left(\prod_{i=n+1}^{\infty}(q^i)^{\lambda(i)} \right)
\]
and 
\[
\mu' = \left(\prod_{i=1}^{n-1}(q^i)^{\mu(i)} \right) (q^n)^{\lambda(n)}  \left(\prod_{i=n+1}^{\infty}(q^i)^{\mu(i)} \right).
\]
Define 
\[
\mathrm{SqDiffNum}(\lambda, \mu) = (\lambda',\mu').
\]
If no such $n$ exists then we say $\mathrm{SqDiffNum}(\lambda,\mu)$ is  not defined.

%
%

\vspace{1cm}
The operation $\mathrm{EulerIdentity}(\lambda,\mu)$, corresponds to the cancellation in the equation 
\[
 \left( \prod_{n=1}^\infty (1-q^n) \right)\left( \prod_{n=1}^\infty (1-q^{2n})^{-1} \right)=\left( \prod_{n=1}^\infty (1-q^{2n-1} \right)
\]
which is related to Euler's Identity of the number of partitions with distinct parts equaling the number of partitions with all odd parts.

$\mathbf{\mathrm{EulerIdentity}(\lambda,\mu)}$. Let $\lambda$ be a partition with distinct parts and let $\mu$ be a partition with all even parts.
Suppose there exists a smallest index $n$ such that 
\[
\lambda(2n) > 0 \text{ or }   \mu(2n) > 0.
\]
If $\lambda(2n) > 0$, then let 
\[
\lambda' =\left( \prod_{i=1}^{2n-1} (q^i)^{\lambda(i)}\right) (q^{2n})^{0}\left( \prod_{i=2n+1}^\infty (q^i)^{\lambda(i)}\right)
\]
and 
\[
\mu =\left( \prod_{i=1}^{2n-1} (q^i)^{\mu_K(i)}\right) (q^{2n})^{\mu(2n)+1}\left( \prod_{i=2n+1}^\infty (q^i)^{\mu(i)}\right).
\]
If $\lambda(2n) = 0$, then let 
\[
\lambda' = \left( \prod_{i=1}^{2n-1} (q^i)^{\lambda(i)}\right) (q^{2n})^{1}\left( \prod_{i=2n+1}^\infty (q^i)^{\lambda(i)}\right)
\]
and 
\[
\mu' = \left( \prod_{i=1}^{2n-1} (q^i)^{\mu(i)}\right) (q^{2n})^{\mu(2n)-1}\left( \prod_{i=2n+1}^\infty (q^i)^{\mu(i)}\right).
\]
Then set 
\[
\mathrm{EulerIdentity}(\lambda,\mu)=(\lambda',\mu').
\]
If there is no such $n$, then we say $\mathrm{EulerIdentity}(\lambda,\mu)$ is not defined. 
$\square$
\end{definition}

\subsection{The matching $\mathcal{T}$}
We describe a matching $\mathcal{T}$ on the set $V_0''(-1)$ that proves the equation
\[
\sum_{v'' \in V_0''(-1)} wt''(v'',q) = \theta_4(q)^2.
\]

\begin{definition}
Let 
\[
v'' = v'(\mu_1, \mu_2, \nu_1, \nu_2, \xi_1, \xi_2, \xi_3,\xi_4, \xi_5, \pi_1,\pi_2,\pi_3,\pi_4, \rho_1, \rho_2, \rho_3,\rho_4) \in V_0''(-1). 
\]
We take the union of $\xi_1, \xi_2$ and $ \xi_3$ and call it a single partition $\mu_3$,
and we rename the partition $\xi_4$ as $\mu_4$. Thus we have four partitions $\mu_k$ for $1\leq k \leq 4$. Now for each $k$ we sum over all partitions $\pi_k,\rho_k$ and $\mu_k$ which gives
\begin{equation} \label{Euler identity}
\prod_{n=1}^\infty (1-q^n)(1-q^n)^{-1} (1+q^n)^{-1}= \prod_{n=1}^\infty (1-q^{2n-1}).
\end{equation}
The above equation follows from Euler's identity equating the number of partitions with distinct parts and the number of partitions with odd parts. 
We define the matching $\mathcal{T}$ to correspond to equation \eqref{Euler identity}.

We take the smallest $K$, $1 \leq K \leq 4$, such that $\mathrm{SqDiffDen}(\rho_K, \mu_K) =(\rho_K', \mu_K')$ is defined. Let  
\[
v''((\rho_K , \mu_K)  \mapsto \mathrm{SqDiffDen}(\rho_K, \mu_K)  )
\]
denote the vertex obtained from $v''$ by replacing $\rho_K$ with $\rho_K'$ and $\mu_K$ with $\mu_K'$. Then we set 
\[
\mathcal{T}(v'') = v''((\rho_K , \mu_K)  \mapsto \mathrm{SqDiffDen}(\rho_K, \mu_K)  ).
\]

If no such $K$ exists, then that means $\rho_k$ is empty and $\mu_k$ is a partition with $\mu(n)$ even for all $n$ and for all $1\leq k \leq 4 $. In this case, we think of each $\mu_k$ as encoding a partition $\sigma_k$ of all even parts via 
\[
\sigma_k = \prod_{i=1}^\infty (q^{2i})^{\frac{\mu_k(i)}{2}}.
\] 
We thus compare $\sigma_k$ and $\pi_k$. Let $K$ be the smallest index $k$ such that $\mathrm{EulerIdentity}(\pi_K, \sigma_K)=(\pi_K', \sigma_K')$ is defined and denote
\[
\mu_K' = \prod_{i=1}^\infty (q^{i})^{2\sigma_K'(2i)}.
\]
Then we set
\[
\mathcal{T}(v'') = v''((\mu_K, \pi_K) \mapsto (\mu_K', \pi_K') ).
\]
If no such $K$ exists, then that means $\mu_k$ and $\rho_k$ are both empty and $\pi_k$ is a partition with distinct odd parts for $1 \leq k \leq 4$. 

Next we take $D$ which is a partition of distinct even parts, and and construct the partition $\frac{1}{2}D$ of distinct arbitrary parts via 
\[
\frac{1}{2}D = \{\frac{d_1}{2}, \frac{d_2}{2}, \frac{d_3}{2}, ..., \frac{d_{|D|}}{2} \}. 
\]
We compare $\frac{1}{2}D $ with the partition $\xi_5$. If $\mathrm{Reciprocal}(\frac{1}{2}D ,\xi_5) = (\frac{1}{2}D' ,\xi_5')$ is defined, then set 
\[
\mathcal{T}(v'') = v''((D,\xi_5)  \mapsto (D',\xi_5') ).
\]
If $\mathrm{Reciprocal}(\frac{1}{2}D ,\xi_5)$ is undefined, then that means both $D$ and $\xi_5$ are empty. 

Next we consider the partitions $A$ and $B$. We transform $A$ into another partition $\tau_1$ via 
\[
\tau_1 =  \prod_{i=1}^{|A|} (q^{\frac{a_i+1}{2}})^1
\]
and transform $B$ into $\tau_2$ via 
\[
\tau_2 =  \prod_{i=2}^{|B|} (q^{\frac{b_i-1}{2}})^1.
\]
Now let $K$ be the smaller index $1\leq K\leq 2$ such that $\mathrm{SqDiffNum}(\tau_K,\nu_K)$ is defined and set 
\[
\mathcal{T}(v'') = v''((\tau_K,\nu_K) \mapsto \mathrm{SqDiffNum}(\tau_K,\nu_K)).
\]
If no such $K$ exists, then that means $\tau_k = \nu_k$ for $1 \leq k \leq 2$ and we think of each pair $(\tau_k,\nu_k)$ as encoding a single partition $\phi_k$ of distinct even parts
\[
\phi_k = \prod_{i=2}^{\infty} (q^{2i})^{\nu(i)}. 
\]
Thus so far on the set $V_0''(-1)$ we have defined the matching $\mathcal{T}$ such that the only vertices $v''$ not in the matching are indexed by 
\[
v'' =( \pi_1,\pi_2, \pi_3,\pi_4,\phi_1,\phi_2 )
\]
such that the $\phi_k$ are partitions with distinct even parts and the $\pi_k$ are partitions with distinct odd parts. We rename these 
\[
(\alpha_1, \beta_1, \delta_1) = ( \pi_1,\pi_2, \phi_1 ) \text{ and } (\alpha_2, \beta_2, \delta_2) = ( \pi_3,\pi_4, \phi_2 ).
\]
Each $(\alpha_k, \beta_k, \delta_k)$ corresponds to a vertex in the original graph $G$. Let $K$ be the smaller index $1\leq K \leq 2$ such that $(\alpha_K, \beta_K, \delta_K)$ is not terminal, if such a $K$ exists. Then set 
\[
\mathcal{T}(v'') = v''((\alpha_K, \beta_K, \delta_K) \mapsto \mathcal{J}((\alpha_K, \beta_K, \delta_K)) ).
\]
This completes the definition of $\mathcal{T}$ on $V_0''(-1)$. 
$\square$
\end{definition}

\begin{lemma}
 The edges $\{  v, \mathcal{T}(v)\}$ for $v \in V_0''(1,-)$ are a matching on $V_0''(1,-)$ and the only vertices not in matching are terminal vertices of the form 
 \[
 v''  = ((\alpha_1, \beta_1, \delta_1) ,(\alpha_2, \beta_2, \delta_2) )
 \]
 in the notation of the definition of $\mathcal{T}$. Also 
 \[
 wt''(\mathcal{T}(v'')) + wt(v'') = 0
 \]
 \end{lemma}
 \begin{proof}
This follows from the construction of $\mathcal{T}$ and the the fact that the four operations on pairs of partitions are sign-changing involutions. $\square$
 \end{proof}

The only vertices $v''$ not in the matching $\mathcal{T}$ are of the form 
\begin{equation} \label{v'' minimal}
v''  = ((\alpha_1, \beta_1, \delta_1) ,(\alpha_2, \beta_2, \delta_2) )
\end{equation}
such that each $(\alpha_k, \beta_k, \delta_k) $ is terminal. For $v''$ of the form \eqref{v'' minimal}, we have 
\[
wt''(v'',q) = wt((\alpha_1, \beta_1, \delta_1),1,-q) wt((\alpha_2, \beta_2, \delta_2),1,-q). 
\]
Since $\theta_4(q) = \theta_3(-q)$, this proves 
\[
\sum_{v'' \in V_0''(-1)} wt''(v'',q) = \theta_4(q)^2.
\]

\subsection{The matching $\mathcal{O}$}
We define a matching $\mathcal{O}$ on the vertices $V_1'' \cup V_2  \subset V''$. 
First we define the operation $\mathrm{Reciprocal}(\lambda,\mu; m)$. 
\begin{definition} 
Let $\lambda$ be a partition with distinct parts and let $\mu$ be any partition. Let $m$ be a positive integer.  
Suppose there exists a smallest index $n\neq m$ such that 
\[
\lambda(n) \neq 0 \text{ or } \mu(n) \neq 0.
\]
If $\lambda(n) \neq 0$, then let 
\[
\lambda' = \left(\prod_{i=1}^{n-1}(q^i)^{\lambda(i)} \right) (q^n)^{0}  \left(\prod_{i=n+1}^{\infty}(q^i)^{\lambda(i)} \right)
\]
\[
\mu' = \left(\prod_{i=1}^{n-1}(q^i)^{\mu(i)} \right) (q^n)^{\mu(n)+1}  \left(\prod_{i=n+1}^{\infty}(q^i)^{\mu(i)} \right).
\]
If $\lambda(n) = 0$, then let 
\[
\lambda' = \left(\prod_{i=1}^{n-1}(q^i)^{\lambda(i)} \right) (q^n)^{1}  \left(\prod_{i=n+1}^{\infty}(q^i)^{\lambda(i)} \right)
\]
and 
\[
\mu' = \left(\prod_{i=1}^{n-1}(q^i)^{\mu(i)} \right) (q^n)^{\mu(n)-1}  \left(\prod_{i=n+1}^{\infty}(q^i)^{\mu(i)} \right).
\]
The define 
\[
\mathrm{Reciprocal}(\lambda,\mu) = (\lambda',\mu').
\]
If no such $n$ exists, then we say that $\mathrm{Reciprocal}(\lambda, \mu; m)$ is not defined. 
\end{definition}

Now any vertex in $\in V_i''(j)$ for $i=1$ or $2$ is indexed by
\[
v'' = (A,B,D;i)(j,\pm)(\mu_1, \mu_2, \nu_1, \nu_2, \xi_1,\xi_2, \xi_3, \xi_4,\xi_5,\pi_1, \pi_2, \pi_3, \pi_4, \rho_1,\rho_2, \rho_3, \rho_4)
\]
where we recall that $A$ and $B$ are partitions of distinct odd parts; $D$ is a partition with distinct even parts; $\mu_1$ and $\mu_2$ are partitions; $\nu_1$ and $\nu_2$ are partitions with distinct parts; $\xi_1$ is a partition with parts congruent to $3 \mod 4$, $\xi_2$ is a partition with parts congruent to $1 \mod 4$,  $\xi_3$ is a partition with even parts, and $\xi_4$ and $\xi_5$ are partitions;  $\pi_k$ are partitions with distinct parts; $\rho_k$ are partitions; and $j$ or $j-1 \in A$ for $+$ and $j$ or $j-1 \in B$ for $-$, or possibly $j=0$ if $i=2$.

\begin{definition}
We define the matching $\mathcal{O}$ on $V_1'' \cup V_2  \subset V''$. 
First take $\xi_4$ and $\xi_5$
\[
\mathcal{O}(v'') = v''((\xi_4,\xi_5) \mapsto \mathrm{SqDiffDen}(\xi_4,\xi_5) ).
\]
If $\mathrm{SqDiffDen}(\xi_4,\xi_5)$ is not defined, then that means $\xi_4$ is empty and we may regard $\xi_5$ as a partition $\chi_1$ with all even parts. Then set 
\[
\mathcal{O}(v'') = v''((\xi_4,\xi_5) \mapsto \mathrm{SqDiffDen}(\xi_3,\chi_1) ).
\]
If $\mathrm{SqDiffDen}(\xi_3,\chi_1)$ is not defined, then that means $\xi_3$ is empty and we may regard $\chi_1$ as a partition $\chi_2$ with all parts congruent to $0 \mod 4$. We then set 
\[
\mathcal{O}(v'') = v''((\lambda,\mu) \mapsto \mathrm{Reciprocal}(\lambda,\mu) )
\] 
where $(\lambda,\mu)$ is the first pair in the following list such that $\mathrm{Reciprocal}(\lambda,\mu)$ is defined: 
\[
(2D,\chi_2), (\nu_1,\mu_1), (\nu_2,\mu_2), (\pi_1,\rho_1), (\pi_2,\rho_2), (\pi_3,\rho_3),(\pi_4,\rho_4).
\]
If $\mathrm{Reciprocal}(\lambda,\mu)$ is not defined for any pair in that list, then we take the partitions $A$ and $B$ and construct the partitions $A_i$ and $B_i$ depending on $i=1$ or $2$: 
\begin{align*}
&A \mapsto A_1&   B\mapsto B_1\\
&a_k \mapsto 2a_k+1&  b_k \mapsto 2b_k-1
\end{align*}
\begin{align*}
&A \mapsto A_2&   B\mapsto B_2\\
&a_k \mapsto 2a_k-1&  b_k \mapsto 2b_k+1
\end{align*}

Let $j' = \lceil \frac{j+1}{2} \rceil$. We then take $(\lambda,\mu)$ to be the first pair in the following lists for which the indicated operation is defined: 

\begin{align*}
i=1, j>0: & \,\,\,\mathrm{Reciprocal}(B_1, \xi_2), \mathrm{Reciprocal}(A_1, \xi_1; 4j'-1)\\
i=1, j<0: & \,\,\,\mathrm{Reciprocal}(A_1, \xi_1), \mathrm{Reciprocal}(B_1, \xi_2; 4j'-3)\\
i=2, j>0: & \,\,\,\mathrm{Reciprocal}(B_2, \xi_1), \mathrm{Reciprocal}(A_2, \xi_2; 4j'-3)\\
i=2, j<0: & \,\,\,\mathrm{Reciprocal}(A_2, \xi_2), \mathrm{Reciprocal}(A_2, \xi_1; 4j'-1)\\
i=2, j=0: & \,\,\,\mathrm{Reciprocal}(A_2, \xi_2), \mathrm{Reciprocal}(B_2, \xi_2).\\
\end{align*}
If none of the above operations are defined, then for $i=1$ that means either: $B$ and $\xi_2$ are empty, and $A_1$and $\xi_1$ are given by 
\begin{align*}
A_1 = (q^{4j-1})^1, \, \, \, \xi_1 = (q^{4j-1})^m
\end{align*}
for some $n$ and $m\geq0$; 
or $A$ and $\xi_1$ are empty, and $B_1$and $\xi_2$ are given by 
\begin{align*}
B_1 = (q^{4j-3})^1, \, \, \, \xi_1 = (q^{4j-3})^m
\end{align*}
for some $m\geq 0$. For $i=2, j \neq 0$ that means
\begin{align*}
A_2 = (q^{4j-3})^1, \, \, \, \xi_2 = (q^{4j-3})^m
\end{align*}
for some $m\geq0$; 
or $A$ and $\xi_2$ are empty, and $B_1$and $\xi_1$ are given by 
\begin{align*}
B_2 = (q^{4j-1})^1, \, \, \, \xi_2 = (q^{4j-1})^m
\end{align*}
for some $m\geq 0$. And for $i=2,j=0$ that means all $A,B, \xi_1,\xi_2$ are empty. 
$\square$
\end{definition}

\begin{lemma}
 The edges $\{  v, \mathcal{O}(v)\}$ for $v \in V_1'' \cup V_2''$ are a matching on those vertices and the only vertices not in matching are vertices indexed by 
 \[
 v  = (d,N; i,j,+) \text{ and }  v  = (d,N; i,j,-)
 \]
where we have  $d\equiv3$ mod 4 with $i=1$ and $+$, or $i=2$ and $-$; $d \equiv 1$ mod 4 with $i=2$ and $+$, or $i=1$ and $-$; and $j=d$ or $j=d+1$. These correspond to the two bits of information $\epsilon_1$ and $\epsilon_2$ in the sets $ \mathrm{Factors3mod4}(n)$ and $ \mathrm{Factors1mod4}(n)$ respectively. And 
 \[
 wt''(\mathcal{O}(v)) + wt''(v) = 0.
 \]
 \end{lemma}
 \begin{proof}
This also follows from the construction of $\mathcal{O}$ and the the fact that the four operations on pairs of partitions are sign-changing involutions. $\square$
 \end{proof}

\section{Paths in Graphs}
Now we can specify the constructive proof. 
We have that he vertices $v \in V_1'' \cup V_2''$ that are not a part of the matching $\mathcal{O}$ correspond to vertices in $ \mathrm{Factors3mod4}(n)$ and $ \mathrm{Factors1mod4}(n)$ and that vertices in $V_0''(1,-)$ that are not a part of the $\mathcal{T}$ matching correspond to the set $ \mathrm{SquarePairs}(n)$. 
We check that the sign of $wt''(v)$ is the same for all $v \in  \mathrm{Factors1mod4}(n)$ and opposite to all those for $v \in  \mathrm{Factors3mod4}(n)\cup \mathrm{SquarePairs}(n)$. We have also shown that the matchings preserve the power of $q$ in the $wt''$ function but change its sign. There facts imply that any path $P$ in $G''$ of the form 
\[
\left( \prod_{k=1}^\infty (\mathcal{J}\mathcal{H})^{m_{3,k}}(\mathcal{J}\mathcal{O})^{m_{2,k}}(\mathcal{J}\mathcal{T})^{m_{1,k}} \right)\mathcal{J}(v)
\]
where $v \in  \mathrm{Factors1mod4}(n)$ for some $n \geq 1$ and the $m_{i,k}$ are non-negative integers is uniquely determined and has its final vertex in $ \mathrm{Factors3mod4}(n)\cup  \mathrm{SquarePairs}(n)$.

\section{Examples}\label{examples}

We present some information about paths for some values of $n$.
When $n=1$, we have four paths which are given in the following table. Since every path $P$ has every other edge in the $\mathcal{J}$ matching, we have omitted those in the table for brevity. 
\begin{table}[H]\caption{Paths for $n=1$}
\centering 
\begin{tabular}{c c c }
\hline \hline                        
1 mod 4 Start vertex &End vertex  &  Edge Pair Path Sequence \\ [0.5ex]
\hline                  
(1,1;1,1,-) & (-1,0) & $\{ \mathcal{T},\mathcal{O}, \mathcal{T},\mathcal{O} \}$  \\
(1,1;1,2,-)& (0,-1)  &$\{ \mathcal{H},\mathcal{T}, \mathcal{O},\mathcal{T}\}$\\
(1,1;2,1,+)  & (0,1)   &$\{ \mathcal{O}, \mathcal{T},\mathcal{O}\}$\\
(1,1;2,2,+)  & (1,0) & $\{ \mathcal{H},\mathcal{T}, \mathcal{O},\mathcal{T},\mathcal{O} \}$\\ [1ex]      
\hline
\end{tabular}\label{table:nonlin}
\end{table}

Below is the table for $n=9$. There are $4(2-1)=4$ ways to write $9$ as a sum of two squares. 
\begin{table}[H]\caption{Paths for $n= 9$}
\centering 
\begin{tabular}{c c c }
\hline \hline                        
1 mod 4 Start vertex &End vertex  & Number of Edge Pairs in Path \\ [0.5ex]
\hline                  
(9,1;1,5,-) & (0,3) & 1,107   \\
(9,1;1,6,-) & (0,-3)  & 6,614\\
(9,1;2,5,+)  & (-3,0)   & 20,638\\
(9,1;2,6,+)  & (3,3;2,2,-) & 15,088\\
(1,9;1,1,-)  & (3,3;2,1,-)  &19,038\\
(1,9;1,2,-)  & (3,0)  & 80,431\\
(1,9;2,1,+)  & (3,3;1,1,+)   &134,951 \\
(1,9;2,1,+)  & (3,3;1,2,+)  & 15,613 \\ [1ex]      
\hline
\end{tabular}\label{table:nonlin}
\end{table}

\section{Further Work} 

\begin{itemize} 

\item The Jacobi Triple Product is an instance of the Macdonald identities. Extend the the proof strategy of Section \ref{Triple Product} to the other Macdonald identities and also see if analogues of the Lambert series \eqref{Lambert} can be obtained. 

\item Analyze the paths in $G''$ more closely and see if they can be expressed using a smaller number of edges by modifying the algorithm $P$. 

\item Use these graphs of partitions to get constructive proofs of the generalizations to the Jacobi sum of two square formula found in \cite{Lass} and \cite{Milne}.

\item Compare the algorithm $P$ to other constructive proofs of two square theorem for primes $p\equiv 1 \mod4$ found in \cite{Brillhart}, \cite{Delorme}, \cite{Elsholtz}, and \cite{Jacobstahl}. 

\end{itemize}


\begin{thebibliography}{9}

\bibitem{Borwein} Borwein, Jonathan M. and Borwein, Peter B.  Pi and the AGM. John Wiley and Sons. New York, 1987. 

\bibitem{Brillhart} Brillhart, John. ``Note on Representing a Prime as a Sum of Two Squares". Mathematics of Computation, Volume 26, Number 120 October 1972. 

\bibitem{Delorme} Delorme, Charles and Pineda-Villavicencio, Guillermo. ``Continuants and Some Decompositions Into Squares." Integers, Vol 15 (2015)

\bibitem{Elsholtz} Elsholtz, Christian. ``A combinatorial approach to sums of two squares and related problems." http://www.math.tugraz.at/~elsholtz/WWW/papers/papers30nathanson-new-address3.pdf 

\bibitem{Jacobi} Jacobi,C.G.J. ``Fundamenta Nova Theoriae Functionum Ellipticarum," Regiomonti. Sumptibus fratrum Bornträger, 1829; reprinted in Jacobi's Gesammelte Werke, Vol. 1, Reimer, Berlin, 1881-1891, pp. 49-239; reprinted by Chelsea, New York, 1969; Now distributed by The American Mathematical Society, Providence, RI

\bibitem{Jacobstahl}  Jacobstahl, Ernst. ``Uber die Darstellung der Primzahlen der Form 4n+1 als Summe zweier Quadrate."
Journal für die reine und angewandte Mathematik / Zeitschriftenband, (1907) Artikel 238-245
http://www.digizeitschriften.de/dms/resolveppn/?PID=GDZPPN002166402

\bibitem{Hirschhorn} Hirschhorn, Michael D. ``A simple proof of Jacobi's two-square theorem," Amer. Math Monthly 92 (1985) 579-580.

\bibitem{Lass}Lass, Bodo. ``Demonstration de la conjecture de Dumont," C. R. Acad. Sci. Paris, Ser. I 341 (2005) 713-718.

\bibitem{Milne} Milne, Stephen S. ``New infinite families of exact sums of squares formulas, Jacobi elliptic functions, and Ramanujan's tau function," PNAS December 24, 1996 93 (26) 15004-15008


\end{thebibliography}
\end{document}